\documentclass[12pt,a4paper,twoside,reqno]{amsart}
\usepackage[utf8]{inputenc}
\usepackage[margin=1in]{geometry} 
\usepackage{amsthm}
\usepackage{amsmath}
\usepackage{amsfonts}
\usepackage{amssymb}
\usepackage{url}
\usepackage[makeroom]{cancel}
\usepackage{enumerate}
\newtheorem{Theorem}{Theorem}

\newtheorem{Corollary}[Theorem]{Corollary}
\newtheorem{Definition}[Theorem]{Definition}
\newtheorem{Lemma}[Theorem]{Lemma}

\DeclareMathOperator*{\E}{\mathbb{E}}
\begin{document}
\title[GCD sums and sum-product estimates]{GCD sums and sum-product estimates}
\author[Thomas F. Bloom]{Thomas F. Bloom}
\address{Heilbronn Institute for Mathematical Research, Department of Mathematics, University of Bristol, 
University Walk, Bristol BS8 1TW, United Kingdom}
\email{matfb@bristol.ac.uk}
\author[Aled Walker]{Aled Walker}
\address{Andrew Wiles Building, University of Oxford, Radcliffe Observatory Quarter, Woodstock Rd, Oxford OX2 6GG, United Kingdom}
\email{walker@maths.ox.ac.uk}
\date{}
\begin{abstract}
In this note we prove a new estimate on so-called GCD sums (also called G\'{a}l sums), which, for certain coefficients, improves significantly over the general bound due to de la Bret\`{e}che and Tenenbaum. We use our estimate to prove new results on the equidistribution of sequences modulo 1, improving over a result of Aistleitner, Larcher, and Lewko on how the metric poissonian property relates to the notion of additive energy. In particular, we show that arbitrary subsets of the squares are metric poissonian.
\end{abstract}
\maketitle

\section{Introduction}
For any two natural numbers $n$ and $m$, let $(n,m)$ denote their greatest common divisor. Let $A$ be a finite set of natural numbers, and let $f:A\longrightarrow \mathbb{C}$ be any function. For $\alpha \in (0,1]$, the so-called GCD sum
\begin{equation}
\label{gcd sum}
\sum\limits_{a,b\in A} \frac{(a,b)^{2\alpha}}{(ab)^\alpha} f(a)\overline{f(b)}
\end{equation}
\noindent has received particular interest (\cite{G49,DH86,ABS15,BS15,BS17,LR17,BT18}), owing to its connections to the resonance method for finding large values of $\zeta(\alpha + it)$ (for instance in \cite{A16, BT18}) and to equidistribution problems (for instance in \cite{ALL16}). 

As mentioned by the authors of \cite{ABS15} and \cite{BT18}, the value $\alpha = 1/2$ represents a critical case, and we shall pay especial attention to this value in our arguments. The best bound in this instance was established by Bondarenko and Seip.

\begin{Theorem}[\cite{BS15}]
\label{Theorem Breteche Tenenbaum}
Let $A \subset \mathbb{N}$ with $\vert A \vert = N$, and let $f:A\longrightarrow \mathbb{C}$ be any function. Then
\[\Big\vert\sum\limits_{a,b\in A} \frac{(a,b)}{\sqrt{ab}} f(a)\overline{f(b)} \Big\vert \leq \exp\left(C\sqrt{\frac{\log N \log\log\log N}{\log \log N}}\right)\| f\|_2^2,\] 
for some absolute constant $C>0$, provided that $N$ is large enough for the logarithms to be defined. 
\end{Theorem}
\noindent Here, and throughout the paper, \[ \Vert f\Vert _ 1 : = \sum\limits_{a\in A} \vert f(a) \vert, \qquad \Vert f\Vert_2^2 : = \sum\limits_{a \in A} \vert f(a) \vert^2.\]

The authors de la Bret\`{e}che and Tenenbaum have recently shown \cite{BT18} that Theorem \ref{Theorem Breteche Tenenbaum} is true with the value $C=2\sqrt{2}+o(1)$, and that $2\sqrt{2}$ is the best possible constant. The purpose of this note is to use results from arithmetic combinatorics, of a `sum-product' flavour, to improve this bound in the special case in which the coefficient function $f$ enjoys some additional structure. We will then given an application of this result to metric number theory.\\

We prove the following general result for $\alpha$ in the range $1/2 \leqslant \alpha \leqslant 1$.

\begin{Theorem}
\label{Theorem general bound off critical line}
Let $f:\mathbb{N} \longrightarrow \mathbb{C}$ be any function with finite support. Let $K$ be some parameter greater than $3$ such that 
\begin{equation}
\label{multiplicative energy bound}
\sum\limits_{ab=cd} f(a)f(b) \overline{f(c)f(d)} \leqslant K\Vert f\Vert_2^4.
\end{equation} Then \[\sum\limits_{a,b} f(a) \overline{f(b)} \frac{(a,b)^{2}}{ab} \ll (\log \log K)^{O(1)} \Vert f\Vert_2^2.\]
\noindent Furthermore, for $\alpha$ in the range $1/2 < \alpha < 1$, \[ \sum\limits_{a,b} f(a) \overline{f(b)} \frac{(a,b)^{2\alpha}}{(ab)^{\alpha}} \ll \exp(O_\alpha( (\log K)^{1-\alpha}(\log\log K)^{-\alpha})) \Vert f\Vert_2^2.\]

\noindent Finally, at the critical value $\alpha=1/2$, if $\Vert f\Vert_1 \geqslant 3$ and $K\geqslant \log \Vert f \Vert_1$ then
\[\sum_{a,b}f(a)\overline{f(b)}\frac{(a,b)}{\sqrt{ab}}\ll (\log \Vert f\Vert _ 1 + O(1))^{-1}\exp(O((\log K\log \log \Vert f\Vert _1)^{1/2}))\|f\|_2^2.\]
\end{Theorem}

One can fruitfully apply these bounds when $f$ has some additive structure, for example when $f(n)=r(n)$, where
\[r(n)=r_A(n)=\lvert \{ (a,b)\in A^2 : a-b=n\}\rvert\]
for some finite set $A\subset \mathbb{N}$. A suitable upper bound on the multiplicative energy of $r(n)$ is provided by the sum-product techniques of arithmetic combinatorics. 

\begin{Theorem}
\label{Theorem multiplicative energy}
If $A\subset\mathbb{N}$ is a finite set and $r=r_A$ is the representation function defined above then
\[\sum_{kl=mn}r(k)r(l)r(m)r(n)\ll \lvert A\rvert^6\log \lvert A\rvert.\]
\end{Theorem}
This theorem is due to Roche-Newton and Rudnev (it is essentially Proposition 4 of \cite{RNR15}) but in section \ref{section sum product} we give a much simpler proof due to Murphy, Roche-Newton, and Shkredov \cite{MRNS15}. 

As an immediate corollary, we obtain the following bound. 

\begin{Theorem}
\label{Theorem half line}

If $A\subset\mathbb{N}$ is a finite set with $\vert A\vert = N$, and $r=r_A$ is the representation function defined above, then, provided $N$ is large enough, \[ \sum\limits_{\substack{n_i,n_j \in A-A\\ n_i,n_j > 0 }} \frac{(n_i,n_j)}{\sqrt{n_i n_j}} r(n_i)r(n_j) \ll (\log N)^{O(1)}\exp(O(( \log\tfrac{N^3}{E(A)}\log\log N)^{1/2}))E(A),\] where \[E(A):=\sum_{n\in \mathbb{Z}} r(n)^2\] is the additive energy of $A$. 
\end{Theorem}

\begin{proof}
We apply the third case of Theorem \ref{Theorem general bound off critical line} to the function $f(n) := r(n)1_{\mathbb{N}}(n)$, and with $K: = C (\log N) (\tfrac{N^3}{E(A)})^2$ for some large constant $C$. Theorem \ref{Theorem multiplicative energy} then implies that condition (\ref{multiplicative energy bound}) is satisfied. Furthermore $\Vert f\Vert_1 = \tfrac{N(N-1)}{2}$, and so $K \geqslant \log \Vert f\Vert_1$. The result then follows. 
\end{proof}

Theorem \ref{Theorem half line} improves over the general bound in Theorem \ref{Theorem Breteche Tenenbaum} in the cases where $E(A) \geqslant N^{3-o(1)}$. We will use this to improve upon a result of Aistleitner-Larcher-Lewko concerning the metric poissonian property, a notion from metric number theory that we will now briefly introduce. For more on this property, see \cite{ALL16, BCGW17, W18}. 

\begin{Definition}
\label{Def metric poiss}
If $\mathcal{A}$ is an increasing sequence of natural numbers, for $\alpha \in [0,1]$ and $s>0$ define \[ F(\alpha,s,N,\mathcal{A}):= \frac{1}{N} \sum\limits_{\substack{x_i,x_j \in A_N \, x_i \neq x_j \\ \Vert \alpha (x_i - x_j) \Vert \leqslant s/N}}1,\]
where $A_N$ is the truncation of $\mathcal{A}$ to the first $N$ elements. We say that $\mathcal{A}$ is \emph{metric poissonian} if for almost all $\alpha \in [0,1]$, for all $s>0$ \[ F(\alpha,s,N,\mathcal{A}) \rightarrow 2s \] as $N\rightarrow \infty$. 
\end{Definition}
\noindent The metric poissonian property is a strong notion of equidistribution for dilates of sequences, motivated by certain concerns in quantum physics (see \cite{RS98}).

We prove the following theorem, continuing the line of work \cite{ALL16, W18, BCGW17, ALT18} that investigates the relationship between the metric poissonian property and the notion of additive energy.  

\begin{Theorem}
\label{Theorem better energy bound for metric poiss}
There exists an absolute positive constant $C$ such that the following is true. Let $\mathcal{A}$ be a sequence of natural numbers, with truncations $A_N$, and suppose that \[E(A_N) \ll N^3/(\log N)^C.\] Then $\mathcal{A}$ is metric poissonian. 
\end{Theorem}
This improves over the result of Theorem 1 of \cite{ALL16}, in which the same conclusion was shown to hold under the stronger hypotheses $E(A_N)\ll N^{3-\delta}$ (for some positive $\delta$). \\

We make no attempt to optimise the value of $C$ in Theorem \ref{Theorem better energy bound for metric poiss}. We have previously speculated that any value of $C$ greater than $1$ should suffice (Fundamental question, \cite{BCGW17}), and this could still be true (it is not precluded by the counterexample constructed in \cite{ALT18}). Unfortunately it does not seem as if the method presented in this paper will be able to prove this result, at least not without substantial modification. 

In \cite{BCGW17} we showed, under some additional and rather stringent density assumptions, that any $C>2$ suffices: in Theorem \ref{Theorem better energy bound for metric poiss} we recover a bound of the same shape without any density assumptions. This enables us to prove the following corollary, answering a question from \cite{ALL16}. 
\begin{Corollary}
\label{Corollary for subsets of squares}
Let $\mathcal{A}$ be an arbitrary infinite subset of the squares. Then $\mathcal{A}$ is metric poissonian.  
\end{Corollary}
\begin{proof}
This follows directly from Theorem \ref{Theorem better energy bound for metric poiss} and the result of Sanders (Theorem 11.7 of \cite{S12}) that if $A_N$ is a set of $N$ squares then $E(A_N) \ll N^3 \exp(-c_1 \log^{c_2} N)$ for some absolute positive constants $c_1$ and $c_2$. 
\end{proof}

\textbf{Notation}: We use the standard Bachmann-Landau asymptotic notation, as well as the Vinogradov symbol $\ll$. Unlike some authors, we use $f \ll g$ to mean that there exists some constant $C$ for which $\vert f(n)\vert \leqslant C \vert g(n)\vert$ for \emph{all} natural numbers $n$, and not just for $n$ sufficiently large. The symbol $f \asymp g$ means that both $f \ll g $ and $g\ll f$ hold. 

\section{Proofs}

The proof of Theorem \ref{Theorem general bound off critical line} will involve the following random model for the zeta function. Let $\{X(p): p \text{ prime} \}$ be a collection of independent random variables, each uniformly distributed on $S^1$. For every $n\in \mathbb{N}$ define $X(n): = \prod\limits_{p^a\Vert n}X(p)^a$. Then define \[\zeta_X(\alpha):= \sum\limits_{n\in\mathbb{N}}\frac{X(n)}{n^\alpha}.\] For fixed $\alpha > 1/2$, this series converges with probability $1$. We will use the following moment estimates from \cite{LR17}.

\begin{Lemma}
\label{Lemma Lewko Radziwill}
When $\zeta_X(\alpha)$ is defined as above, and $l$ is a real number, we have 
\begin{equation}
\log\E\vert \zeta_X(\alpha) \vert ^{2l} \ll \begin{cases} 
l\log\log l & \alpha = 1 \\
C_\alpha l^{1/\alpha}(\log l)^{-1} & 1/2 < \alpha < 1 \\
l^2 \log (\alpha - 1/2)& 1/2 < \alpha \end{cases}
\end{equation}
for some positive constant $C_\alpha$. The first two cases hold provided $l\geqslant 3$, and the final case holds provided $l\geqslant 1$.
\end{Lemma}
\begin{proof}
This is Lemma 6 of \cite{LR17}. The range of uniformity in $l$ is not specified there, but it is quickly seen from their argument that the above ranges of $l$ are acceptable. 
\end{proof}

\begin{proof}[Proof of Theorem \ref{Theorem general bound off critical line}]
We use the ideas of Lewko and Radziwill from \cite{LR17}. Let \[ D(X) : = \sum\limits_{n \in \mathbb{N}} f(n) X(n) .\] Then on the one hand 
\begin{align*}
\mathbb{E} (\vert \zeta_X(\alpha) D(X)\vert^2 )
&= 
\sum_{n_1,n_2,m_1,m_2}(n_1n_2)^{-\alpha}f(m_1)\overline{f(m_2)}1_{n_1m_1=n_2m_2}\\
&=
\sum_n\Big\lvert\sum_{m\mid n}f(m)(m/n)^\alpha\Big\rvert^2\\
&=
\zeta(2\alpha) \sum\limits_{a,b} f(a) \overline{f(b)} \frac{(a,b)^{2\alpha}}{(ab)^\alpha}.
\end{align*}
On the other hand, by splitting the expectation according to whether $\lvert \zeta_X(\alpha)\rvert<V$, and using the identity $\mathbb{E}(\lvert D(X)\rvert^2)=\Vert f\Vert_2^2$,  for all positive $l$ and $V$ we have 
\begin{equation}
\label{splitting expectation}
\mathbb{E} (\vert \zeta_X(\alpha) D(X)\vert^2) \leqslant V^2\Vert f\Vert_2^2 + V^{-2l} \mathbb{E} (\vert \zeta_X(\alpha)\vert ^{2l + 2} \lvert D(X)\rvert^2).
\end{equation} Now our approach differs from \cite{LR17}. Instead of removing $D(X)$ from the second summand using an $L^\infty$ bound, we use the Cauchy-Schwarz inequality. This shows that (\ref{splitting expectation}) is at most
\begin{equation}
\label{changing to fourth moment}
V^2 \Vert f\Vert_2^2 + V^{-2l}\mathbb{E} (\vert D(X)\vert^4)^{\frac{1}{2}} \mathbb{E} ( \vert \zeta_{X}(\alpha)\vert^{4l+4})^{\frac{1}{2}}.
\end{equation}

Expanding out the expectation we see that \[\mathbb{E} (\vert D(X)\vert^4) = \sum\limits_{ab=cd} f(a)f(b) \overline{f(c)f(d)}\leq K\| f\|_2^4.\]
Suppose first that $\alpha = 1$. Then, by Lemma \ref{Lemma Lewko Radziwill}, the bound in (\ref{changing to fourth moment}) is \[\ll \Vert f\Vert_2^2(V^2  + V^{-2l}K^{\frac{1}{2}}\exp(O(l \log\log l))).\] 
Choosing $V = (\log l)^C$, with $C$ large enough, and $l=\log K+3$, this is 
\[ \ll(\log \log K)^{O(1)} \Vert f\Vert_2^2\] as claimed. (Note that, since $K \geqslant 3$ by assumption, $\log l \ll \log \log K$).

When $1/2 <\alpha < 1$, (\ref{changing to fourth moment}) enjoys the bound \[\Vert f\Vert_2^2(V^2  + V^{-2l}K^{\frac{1}{2}}\exp(O_\alpha(l^{1/\alpha} (\log l)^{-1}))).\] Picking $V = \exp(C_\alpha l^{-1 + 1/\alpha}(\log l)^{-1}))$, with $C_\alpha$ large enough, and $l = (\log K)^\alpha (\log \log K)^{\alpha} + 3$ yields the result.

It remains to tackle the case $\alpha=1/2$, which we do by interpolating from $\alpha>1/2$. More precisely, let $\alpha=1/2+1/\log \| f\|_1.$ By H\"{o}lder's inequality
\[\sum_{a,b}\frac{(a,b)}{\sqrt{ab}}f(a)\overline{f(b)}\ll \left(\sum_{a,b}\frac{(a,b)^{2\alpha}}{(ab)^\alpha}f(a)\overline{f(b)}\right)^{1/2\alpha}.\]
Using the estimate $\zeta(s)=(s-1)^{-1}+O(1)$ and (\ref{changing to fourth moment}) gives
\[\sum_{a,b}\frac{(a,b)^{2\alpha}}{(ab)^\alpha}f(a)f(b)\ll  (\log \| f\|_1 + O(1))^{-1}\| f\|_2^2( V^2+V^{-2l}K^{1/2}\exp(O(l^2\log\log \|f\|_1))).\]
We now choose $V=\exp(Cl(\log\log\|f\|_1))$, with $C$ large enough, and \\$l=(\log K)^{1/2}(\log \log \|f\|_1)^{-1/2}$. Note that the assumption $K\geqslant \log \Vert f\Vert_1$ implies that $l\geqslant 1$, and so our application of Lemma \ref{Lemma Lewko Radziwill} was valid. The result then follows.  
\end{proof}

\begin{proof}[Proof of Theorem \ref{Theorem better energy bound for metric poiss}]
Fixing two positive values $s$ and $\varepsilon$, it will be enough to show that for almost all $\alpha$ there exists a natural number $N(\alpha,s,\varepsilon)$ such that 
\begin{equation}
\label{desired bound}
\vert F(\alpha,s,N,\mathcal{A}) - 2s\vert < \varepsilon
\end{equation} for every $N> N(s,\alpha,\varepsilon)$. (The theorem then follows by considering all rational values of $s$ and $\varepsilon$: this argument is also used to conclude section 6 of \cite{BCGW17}).

We have the following lemma of Aistleitner-Larcher-Lewko, which is essentially Lemma 3 of \cite{ALL16}. 
\begin{Lemma}
If $s \asymp 1$ then
\[\operatorname{Var}(F,N) : = \int\limits_{0}^1 \vert F(\alpha,s,N,\mathcal{A}) - 2s \vert^2 \, d\alpha \ll \frac{ \log N}{N^3} \sum\limits_{\substack{n_i,n_j \in A_N - A_N \\ n_i,n_j >0}} \frac{(n_i,n_j)}{\sqrt{n_in_j}} r(n_i)r(n_j).\]
\end{Lemma}
\noindent Applying Theorem \ref{Theorem half line}, this implies that \[ \operatorname{Var}(F,N) \ll  (\log N)^{O(1)}\exp(O( \log \tfrac{N^3}{E(A)}\log\log N)^{1/2}) \frac{E(A)}{N^3}.\] Therefore, assuming $\tfrac{E(A)}{N^3}\leq (\log N)^{-C}$ for some sufficiently large absolute constant $C>0$,
\[ \operatorname{Var}(F,N) \ll (\log N)^{-3/2}. \]

We may now conclude in a similar fashion as in \cite{ALL16} and \cite{BCGW17}. Indeed, let $\eta > 0$ be a small positive quantity to be chosen later, and for $j\in \mathbb{N}$ let $N_j : = \lfloor e^{\eta j} \rfloor$. With this choice, $ \sum_j \operatorname{Var}(F,N_j) < \infty$. By Chebyschev's inequality and the Borel-Cantelli Lemma, this implies that for almost all $\alpha$ there exists a value $j(\alpha,s,\varepsilon)$ such that for all $j\geqslant j(\alpha,s,\varepsilon)$ one has 
\begin{equation}
\label{equation j bound}
\vert F(\alpha,s\tfrac{N_{j}}{N_{j+1}},N_j,\mathcal{A}) - 2s\vert <  \varepsilon/2, \qquad \vert F(\alpha,s\tfrac{N_{j+1}}{N_j},N_j,\mathcal{A}) - 2s\vert < \varepsilon / 2.
\end{equation}
\noindent Let $N$ be large, and choose $j$ such that $N_j < N \leqslant N_{j+1}$. Then \[N_j F(\alpha,s\tfrac{N_{j}}{N_{j+1}},N_j,\mathcal{A}) \leqslant N F(\alpha,s,N,\mathcal{A})\leqslant N_{j+1} F(\alpha,s\tfrac{N_{j+1}}{N_j},N_j,\mathcal{A}). \] Using this sandwiching, and the fact that $N_{j+1}/N_j = 1 + O(\eta)$, one may establish from (\ref{equation j bound}) that, if $N$ is large enough so that $j \geqslant j(\alpha,s,\varepsilon)$, \[\vert F(\alpha,s,N,\mathcal{A}) - 2s\vert \leqslant \varepsilon/2 + O_s(\eta).\] If $\eta$ is small enough, (\ref{desired bound}) follows.
\end{proof}

\section{Sum-product estimates}
\label{section sum product}
This section is devoted to presenting a proof of Theorem \ref{Theorem multiplicative energy}. The first proof used the deep methods of Guth and Katz, but subsequently Murphy, Roche-Newton, and Shkredov gave a simpler proof using only the Szemer\'{e}di-Trotter theorem (see Theorem 2.1 of \cite{MRNS15}). Since the proof is short and elementary, but rather difficult to extract from \cite{MRNS15}, we include a proof here.

We will require the following simple consequence of the Szemer\'{e}di-Trotter theorem (Corollary 2.2 of \cite{MRNS15}).

\begin{Theorem}\label{sztr}
Let $t\geqslant 2$ be a parameter. If $L$ is a finite set of lines in $\mathbb{R}^2$ such that at most $O(\lvert L\rvert^{1/2})$ such lines intersect at any point then the number of points on more than $t$ lines is $O(\lvert L\rvert^2 /t^3)$.
\end{Theorem}

We use this to prove the following result.
\begin{Lemma}
\label{Lemma application of szem trot}
For any finite sets $A,Z\subset \mathbb{R}$
\[\sum_{n,m}r(n)r(m)1_{n/m\in Z}\ll \lvert A\rvert^3\lvert Z\rvert^{1/2}.\]
\end{Lemma}
\begin{proof}
Let $r(y,z)$ count the number of $a,b\in A$ such that $a+bz=y$. Then 
\[\sum_{n,m}r(n)r(m)1_{n/m\in Z}=\sum_{z\in Z}\sum_y r(y,z)^2.\]
By Theorem~\ref{sztr} the number of $(y,z)$ such that $r(y,z)\geq t$ is $O( \lvert A\rvert^4/t^3)$. With $u$ some parameter to be chosen later, splitting the sum by whether $r(y,z)\geq u$ and dividing the large values dyadically, we have 
\begin{align*}
\sum_{z\in Z}\sum_{y}r(y,z)^2
&\ll u\lvert Z\rvert\lvert A\rvert^2 + \sum_{2^i\geq u}\frac{\lvert A\rvert^4}{2^{i}}\\
&\ll u\lvert Z\rvert\lvert A\rvert^2+	u^{-1}\lvert A\rvert^4.
\end{align*}
The lemma follows by choosing $u\asymp \lvert A\rvert\lvert Z\rvert^{-1/2}$.
\end{proof}

\begin{proof}[Proof of Theorem \ref{Theorem multiplicative energy}]
This follows from Lemma \ref{Lemma application of szem trot} following another dyadic decomposition. Indeed, since one has the trivial bound $r(k) \leqslant \vert A\vert$ for all $k$, \begin{align}
\label{dyadic decomp}
\sum\limits_{km=ln}r(k)r(l)r(m)r(n) &= \sum\limits_z \Big\vert \sum\limits_{n/m = z} r(n)r(m) \Big\vert^2 \nonumber\\
& = \sum\limits_{i=1}^{ \lceil 2\log \vert A\vert \rceil} \sum\limits_{ z \in Z_i} \Big\vert \sum\limits_{n/m = z} r(n)r(m) \Big\vert^2 ,
\end{align}
\noindent where \[ Z_i : = \{ z : e^{i-1}\leqslant \Big\vert \sum\limits_{n/m = z} r(n)r(m) \Big\vert < e^i\}.\] So (\ref{dyadic decomp}) is at most \[\sum\limits_{i=1}^{ \lceil 2\log \vert A\vert \rceil}\vert Z_i\vert e^{2i}.\] But applying Lemma \ref{Lemma application of szem trot} to $Z_i$ yields $\vert Z_i\vert \ll \vert A\vert^6 e^{-2i}.$ So \[\sum\limits_{km=ln}r(k)r(l)r(m)r(n) \ll \vert A\vert^6 \log \vert A\vert,\] as desired.
\end{proof}

\section{Final remarks}
Considering the improvements of Bondarenko and Seip on GCD sums with $\alpha = 1/2$, we do not think that Theorem \ref{Theorem half line} has the quantitatively optimal form. However, to offer a substantial improvement we suspect that one would have to combine the incidence geometry argument of this paper with the compression arguments used in \cite{BS15} and \cite{BT18}, or find some other means to take simultaneous advantage of both the additive and multiplicative structures of the sums involved. We have not been able to see a way of doing this, and indeed in our argument above we separate the additive and multiplicative structures early on, through applying Cauchy-Schwarz to (\ref{splitting expectation}). 

Though in Theorem \ref{Theorem better energy bound for metric poiss} we have refined the state of knowledge about the relationship between the metric poissonian property and additive energy, it is still unclear what the truth should be. We know from \cite{ALT18} that there is no sharp threshold phenomenon, but we still believe that Theorem \ref{Theorem better energy bound for metric poiss} is true for any $C$ greater than $1$, which, up to $(\log N)^{o(1)}$ factors, would be an optimal result. \\

\bibliographystyle{plain}
\bibliography{gcdsumimprovement.bib}

\end{document}